\pdfoutput=1
\documentclass[11pt, reqno]{amsart}

\numberwithin{equation}{section}

\usepackage{amsmath,amsfonts,amssymb}
\usepackage{mathtools}
\usepackage[utf8]{inputenc}
\usepackage[english]{babel}
\usepackage{hyperref}
\usepackage{a4wide}
\usepackage{cite}
\usepackage{stmaryrd}
\usepackage[new]{old-arrows}

\DeclareUnicodeCharacter{FB01}{fi}
\newtheorem{theorem}{Theorem}[section]

\newtheorem{lemma}[theorem]{Lemma}

\newtheorem*{aim-non}{Aim}

\newtheorem*{conjecture-non}{Conjecture}

\theoremstyle{definition}
\newtheorem{definition}{Definition}

\begin{document}
	\title[Semiprojectivity of principal $G$-bundles with $\lambda$-connections]
	{Semiprojectivity of the moduli of principal $G$-bundles with $\lambda$-connections}
	
	\author[S. Roy]{Sumit Roy}
 \address{Stat-Math Unit, Indian Statistical Institute, 203 B.T. Road, Kolkata 700 108, India.}
 \email{sumit.roy061@gmail.com}

 \author[A. Singh]{Anoop Singh}
\address{Department of Mathematical Sciences, Indian Institute of Technology (BHU), Varanasi- 221 005, India.}
\email{anoopsingh.mat@iitbhu.ac.in}

\subjclass[2020]{14D20, 14D23, 70G45, 14H60}
 \keywords{Principal bundles, Higgs bundles, Holomorphic connections, Semiprojective, Białynicki–Birula decomposition}
	\begin{abstract}
	    Let $X$ be a compact connected Riemann surface of genus $g \geq 2$ and $G$  a connected reductive affine algebraic group over $\mathbb{C}$. We prove the semiprojectivity of the moduli spaces of semistable $G$-Higgs bundles and $G$-bundles with $\lambda$-connections of fixed topological type $d\in \pi_1(G)$. As an application, in the smooth case we describe the resulting Białynicki–Birula decomposition and derive cohomological and motivic consequences.
	\end{abstract}
	\maketitle
	
	\section{Introduction}
	Let $X$ be a connected compact Riemann surface of genus $g \geq 2$. Let $G$ be a connected reductive affine algebraic group over $\mathbb{C}$. In this article, we consider the moduli space $\mathcal{M}^d_{\mathrm{Higgs}}(G)$ (resp. $\mathcal{M}^d_{\mathrm{conn}}(G)$) of semistable $G$-Higgs bundles (resp. holomorphic $G$-connections) of fixed topological type $d \in \pi_1(G)$ over the curve $X$. These two moduli spaces are not smooth. But if we consider the regularly stable locus (i.e. those elements for which the automorphism group coincides with the center of $G$), then the moduli spaces are smooth. 

 	Simpson in \cite{S94a} considered a family over $\mathbb{C}$, called the Hodge moduli space, whose fibers over $0$ and $1$ are exactly the moduli of semistable Higgs bundles and of holomorphic connections respectively. Also, he produced a homeomorphism between the moduli space of semistable Higgs bundles and the moduli space of holomorphic connections, which is known as the non-abelian Hodge correspondence (see \cite{S92}, \cite{S94a}, \cite{S94}). In general,  these two moduli spaces have singularities but if we consider the case where the rank and degree are coprime, then they are smooth.


	Let $V$ be a quasi-projective variety over $\mathbb{C}$, equipped with a $\mathbb{C}^*$-action $v \mapsto t\cdot v$, $v\in V, t \in \mathbb{C}^*$. We call that $V$ is \textit{semiprojective} if it satisfies:
	\begin{enumerate}
	    \item for all $v \in V$, the limit $$\lim_{t\to 0} (t\cdot v) \in V$$ exists in $V$,
	    \item the fixed point subvariety $V^{\mathbb{C}^*} \subset V$ is proper in $V$.
\end{enumerate}

The moduli space $\mathcal{M}^d_{\mathrm{Higgs}}(G)$ admits a standard $\mathbb{C}^*$-action, where scaling the Higgs field by any $t \in \mathbb{C}^*$ preserves semistability and stability of the bundle. The Hitchin map $$h : \mathcal{M}^d_{\mathrm{Higgs}}(G) \to \mathcal{H}\coloneqq \bigoplus_{i=1}^{r}\mathrm{H}^0(X, K_X^{d_i}),$$ where $r=\mathrm{rank}(G)$, is shown to be $\mathbb{C}^*$-equivariant, mapping Higgs bundles to invariant polynomials on $\mathfrak{g}$, the Lie algebra of $G$. Using the properness of the Hitchin map we show that for a semistable $G$-Higgs bundle $(E_G,\varphi)$, the limit $\lim_{t\to 0} (E_G,t\varphi)$ exists in $\mathcal{M}^d_{\mathrm{Higgs}}(G)$. We then show that the fixed points under the $\mathbb{C}^*$-action are contained within $h^{-1}(0)$ and are shown to be proper. This is important because the origin is the only fixed point in the Hitchin base under the $\mathbb{C}^*$-action. This proves that the moduli space $\mathcal{M}^d_{\mathrm{Higgs}}(G)$ is semiprojective (see Theorem \ref{Higgs}).

Similarly, we show that the principal Hodge moduli space $\mathcal{M}^{d}_{\mathrm{Hod}}(G)$ is semiprojective by showing that the limit $\lim_{t\to 0} (E_G,t\lambda,t\nabla)$ exists in $\pi^{-1}(0) \subset \mathcal{M}^d_\mathrm{Higgs}(G)$, where
\begin{align*}
\begin{split}
\pi: \mathcal{M}^{d}_\mathrm{Hod}(G) &\longrightarrow \mathbb{C}\\
(E,\lambda,\nabla) &\longmapsto \lambda,
\end{split}
\end{align*}
is the projection map (see Theorem \ref{Hodge}).

Finally, we provide two structural consequences of the semiprojectivity of $M^d_{\mathrm{Higgs}}(G)$. First, semiprojectivity allows one to apply the
Białynicki--Birula theory to the $\mathbb{C}^\ast$-action. As a result, the moduli space admits a decomposition into attracting sets indexed by the connected components of the fixed point locus. This provides a geometric stratification in which each stratum is organized around a fixed component. In particular, the global geometry of the moduli space is controlled by the geometry of its fixed point locus together with the dimensions of the corresponding attracting cells.

Second, this stratification has direct cohomological and motivic implications. At the level of compactly supported cohomology, the space decomposes additively as a direct sum of degree shifts of the cohomology of the fixed components. Thus the topology of the moduli space is determined by the
fixed point geometry and the combinatorics of the attracting
cells.

Similarly, in the Grothendieck ring of varieties,
the class of $M^d_{\mathrm{Higgs}}(G)$ can be written as a finite
sum of powers of the Lefschetz motive multiplied by the classes
of the fixed components.
In this sense, the motive of the moduli space is entirely
expressed in terms of its fixed locus.

These consequences show that semiprojectivity is not only a
formal property of the $\mathbb{C}^\ast$-action, but a geometric
structure that governs both the topology and the motive of the
moduli space.

 \section{Preliminaries}
	Let $K_X$ denote the holomorphic cotangent bundle on $X$. Let $G$ be a connected reductive affine algebraic group over $\mathbb{C}$ and let $\mathfrak{g}=\mathrm{Lie}(G)$ be the Lie algebra of $G$. The adjoint action of $G$ on $\mathfrak{g}$ is denoted by
	\begin{align*}
	    	\mathrm{ad} : G \longrightarrow \operatorname{End}(\mathfrak{g}).
      \end{align*}

\begin{definition}
     A \textit{holomorphic principal $G$-bundle} over $X$ is a holomorphic fiber bundle $E_G$, i.e. there is a surjective holomorphic map $p : E_G \to X$ and a holomorphic right action $\phi: E_G \times G \longrightarrow E_G$ of $G$ on $E_G$ such that the following conditions hold:
\begin{enumerate}
    \item $p \circ \phi = p \circ p_1$,  where  $p_1 : E_G \times G \longrightarrow E_G$ is the projection map, and 
    \item the map
    \begin{align*}
        E_G \times G &\longrightarrow E_G \times_X E_G \\
        (y,g) &\mapsto (y,\phi(y,g))
    \end{align*}
 to the fiber product is a biholomorphism.
    
\end{enumerate}
\end{definition}

The right action of $G$ on $E_G$ together with the adjoint action of $G$ on $\mathfrak{g}$ gives a $G$-action on $E_G \times \mathfrak{g}$ defined by
	\[
	(v,\xi)\cdot g = (v\cdot g, \mathrm{ad}(g^{-1})(\xi)), \hspace{0.2cm} \forall \hspace{0.1cm} (v,\xi)\in E_G\times \mathfrak{g},\hspace{0.1cm} g\in G.
	\] 
	The associated quotient bundle $$E_G \times^G \mathfrak{g} \coloneqq (E_G \times \mathfrak{g})/G$$ is called the \textit{adjoint vector bundle} of $E_G$ and it is denoted by $\mathrm{ad}(E_G)$. The topological type of a holomorphic principal $G$-bundle $E_G$ over $X$ corresponds to an element of the fundamental group $\pi_1(G)$ (see \cite{R75}) and this is a finitely generated abelian group. 
 
\begin{definition}
     A holomorphic principal $G$-bundle $E_G$ is called \textit{stable} (respectively, \textit{semistable}) if for all maximal parabolic subgroup $P \subset G$ and every holomorphic reduction $E_P$ of the structure group of $E_G$ to $P$,
	\[
	\deg(\mathrm{ad}(E_P)) < 0 \hspace{0.2cm} (\mathrm{respectively,} \hspace{0.2cm} \leq 0 \hspace{0.05cm})
	\]
	where $\mathrm{ad}(E_P) \subset \mathrm{ad}(E_G)$ is the adjoint vector bundle of $E_P$.  
\end{definition}	
\begin{definition}	
A stable principal $G$-bundle $E_G$ is called \textit{regularly stable} if $\mathrm{Aut}(E_G)= Z(G)$, i.e. the automorphism group of $E_G$ coincides with the center of $G$.
\end{definition}

Let $\mathcal{M}^d(G)$ denote the moduli space of semistable holomorphic $G$-bundles over $X$ of topological type $d\in \pi_1(G)$. It is well known that the moduli space $\mathcal{M}^d(G)$ is an irreducible normal projective complex variety of dimension
\[
\dim \mathcal{M}^d(G) = (g-1)\cdot \dim_{\mathbb{C}}G + \dim_\mathbb{C} Z(G),
\]
 (see \cite{R75}, \cite{R96} for more details). The moduli space $$\mathcal{M}^{d,rs}(G) \subset \mathcal{M}^d(G)$$ of regularly stable principal $G$-bundles is an open subvariety and is exactly the smooth locus of $\mathcal{M}^d(G)$ (see \cite[Corollary $3.4$]{BH12}).

\subsection{$G$-Higgs bundles}
\begin{definition}
    A principal \textit{$G$-Higgs bundle} over $X$ is a pair $(E_G,\varphi)$ where $E_G$ is a holomorphic principal $G$-bundle and 
    \[
    \varphi \in \mathrm{H}^0(X, \mathrm{ad}(E_G) \otimes K_X)
    \]
    is a holomorphic section, called the \textit{Higgs field} \cite{H87, S92}. 
\end{definition}

\begin{definition}
    A principal \textit{$G$-Higgs bundle} $(E_G,\varphi)$ is called \textit{stable} (respectively, \textit{semistable}) if for all holomorphic reduction $E_P$ of the structure group of $E_G$ to a $\varphi$-invariant maximal parabolic subgroup $P\subsetneq G$, i.e.  $\varphi \in \mathrm{H}^0(X, \mathrm{ad}(E_P) \otimes K_X)$ we have
\[
	\deg(\mathrm{ad}(E_P)) < 0 \hspace{0.2cm} (\mathrm{respectively,} \hspace{0.2cm} \leq 0 \hspace{0.05cm}).
\]
\end{definition}

Let $\mathcal{M}^d_{\mathrm{Higgs}}(G)$ denote the moduli space of semistable principal $G$-Higgs bundles over $X$ of topological type $d\in \pi_1(G)$. Following \cite{S94}, we know that $\mathcal{M}^d_{\mathrm{Higgs}}(G)$ is a normal irreducible quasi-projective variety over $\mathbb{C}$ of dimension
\[
\dim \mathcal{M}^d_{\mathrm{Higgs}}(G) = 2\dim \mathcal{M}^d(G)= 2(g-1)\cdot \dim_{\mathbb{C}}G + 2 \dim_\mathbb{C} Z(G).
\]
Observe that $\mathcal{M}^d(G) \subset \mathcal{M}^d_{\mathrm{Higgs}}(G)$ is closed subvariety of $\mathcal{M}^d_{\mathrm{Higgs}}(G)$ via the embedding
\begin{align*}
\mathcal{M}^d(G) &\longhookrightarrow \mathcal{M}^d_{\mathrm{Higgs}}(G) \\
E_G &\longmapsto (E_G,0).
\end{align*}	
There is a natural $\mathbb{C}^*$-action on $\mathcal{M}^d_{\mathrm{Higgs}}(G)$ given by
\begin{equation}\label{actionHiggs}
    t \cdot (E_G,\varphi) \coloneqq (E_G, t\varphi).
\end{equation}

From the deformation theory, the tangent space of $\mathcal{M}^{d,rs}(G)$ at $E_G$ is isomorphic to $\mathrm{H}^1(X, \mathrm{ad}(E_G))$. By Serre duality, we have
\[
\mathrm{H}^0(X, \mathrm{ad}(E_G) \otimes K_X) \cong \mathrm{H}^1(X, \mathrm{ad}(E_G))^*.
\]
Thus the cotangent bundle of $\mathcal{M}^{d,rs}(G)$, $$T^*\mathcal{M}^{d,rs}(G) \subset \mathcal{M}^{d}_{\mathrm{Higgs}}(G)$$ is an open subvariety of $\mathcal{M}^{d}_{\mathrm{Higgs}}(G)$.
	
\subsection{Holomorphic $G$-connections}
 Let $p$ denote the projection morphism from the total space of $E_G$ to $X$. For any open subset $U \subset X$, let $\mathcal{A}(U)$ denote the space of $G$-equivarient holomorphic vector fields on $p^{-1}(U)$. Let $\mathcal{A}$ be the coherent sheaf on $X$ which associates to any $U$ to the vector space $\mathcal{A}(U)$. The corresponding vector bundle is called the \textit{Atiyah bundle} for $E_G$ and it is denoted by $\mathrm{At}(E_G)$ (see \cite{A57}). In fact, it is given by the quotient
 \[
 \mathrm{At}(E_G) \coloneqq (TE_G)/G
 \]
 where $TE_G$ is the holomorphic tangent bundle of $E_G$; so $\mathrm{At}(E_G)$ is a holomorphic vector bundle over $E_G/G = X$. Consequently, we have an exact sequence of vector bundles
	\begin{equation}\label{Atiyah}
	0 \longrightarrow \mathrm{ad}(E_G) \longrightarrow \mathrm{At}(E_G) \overset{\eta}\longrightarrow TX \longrightarrow 0,
	\end{equation}
	where $TX$ is the holomorphic tangent bundle of $X$. The morphism $\eta$ is defined using the differential $dp$ of $p: E_G \to X$. Also, note that the adjoint bundle $\mathrm{ad}(E_G)$ is the subbundle of the tangent bundle $TE_G$ defined by the kernel of $dp$. The above short exact sequence (\ref{Atiyah}) is known as the \textit{Atiyah exact sequence} for the principal $G$-bundle $E_G$.
	
	A \textit{holomorphic connection} on $E_G$ is a holomorphic splitting of the Atiyah exact sequence, i.e., a holomorphic homomorphism
	\[
	\mathcal{D} : TX \longrightarrow \mathrm{At}(E_G)
	\]
	such that $$\eta \circ \mathcal{D} = \mathrm{id}_{TX}$$ for the morphism $\eta$ in the Atiyah sequence (\ref{Atiyah}). If $\mathcal{D}'$ is an another splitting of (\ref{Atiyah}), then $\mathcal{D}-\mathcal{D}'$ is a holomorphic homomorphism from $TX$ to $\mathrm{ad}(E_G)$. Conversely, for any holomorphic section $s \in \mathrm{H}^0(X, K_X \otimes \mathrm{ad}(E_G))$, if $\mathcal{D}$ is a splitting of (\ref{Atiyah}) then so is $\mathcal{D}+s$. Therefore, the space of all holomorphic connections on $E_G$ is an affine space for the vector space $\mathrm{H}^0(X, K_X \otimes \mathrm{ad}(E_G))$. 
	
	Since $X$ has complex dimension one, any holomorphic connection on $E_G$ is automatically a flat holomorphic connection on $E_G$ compatible with its holomorphic structure and since $\mathrm{H}^1(X, K_X \otimes \mathrm{ad}(E_G))$ parametrizes the space of all extensions of $TX$ by $\mathrm{ad}(E_G)$, the condition required for the existence of a flat holomorphic connection on the principal bundle $E_G$ is equivalent to the condition that if $\alpha \in \mathrm{H}^1(X, K_X \otimes \mathrm{ad}(E_G))$ corresponds to the sequence (\ref{Atiyah}) then $\alpha=0$ (see \cite{AB02}). 
	
	\begin{definition}
		A \textit{holomorphic $G$-connection} is a pair $(E_G,\mathcal{D})$ where $E_G$ is a holomorphic principal $G$-bundle and $\mathcal{D}$ is holomorphic connection on $E_G$.
	\end{definition}
	A holomorphic connection on a principal $G$-bundle implies semistability (see \cite{AB02}). Let $\mathcal{M}^d_\mathrm{conn}(G)$ denote the moduli space of holomorphic $G$-connections over $X$ of fixed topological type $d\in \pi_1(G)$. By \cite{BGH13}, the moduli space $\mathcal{M}^d_\mathrm{conn}(G)$ is a normal irreducible quasi-projective variety over $\mathbb{C}$ of dimension $$\dim \mathcal{M}^d_{\mathrm{conn}}(G)= \dim \mathcal{M}^d_{\mathrm{Higgs}}(G)= 2(g-1)\cdot \dim_{\mathbb{C}} G.$$

	\subsection{$\lambda$-connections}
		Let $p : E_G \to X$ be a holomorphic principal $G$-bundle over $X$ and let $\lambda \in \mathbb{C}$.
		
		\begin{definition}
		A $\lambda$\textit{-connection} on $E_G$ over $X$ is a holomorphic map of vector bundles
		\[
		\nabla : TX \longrightarrow \mathrm{At}(E_G)
		\]
		such that $\eta \circ \nabla = \lambda\cdot \mathrm{id}_{TX}$ for the morphism $\eta$ in the Atiyah sequence (\ref{Atiyah}).
		\end{definition}
		If $\nabla$ is a $\lambda$-connection on $E_G$ with $\lambda \ne 0$, then $\lambda^{-1}\nabla$ is a holomorphic $G$-connection on $E_G$. Therefore, $(E_G,\nabla)$ is automatically semistable for $\lambda \ne 0$. 
		
		Let $\mathcal{M}^d_\mathrm{Hod}(G)$ be the moduli space consisting of triples $(E_G,\lambda,\nabla)$, where $\lambda \in \mathbb{C}$, $E_G$ is a principal $G$-bundle over $X$ of topological type $d\in \pi_1(G)$ and $\nabla$ is a semistable $\lambda$-connection on $E_G$ (see \cite{S94}, \cite{BGH13} for details).
	
	There is a canonical surjective algebraic map 
	\begin{align}\label{proj}
 \begin{split}
	\pi: \mathcal{M}^d_\mathrm{Hod}(G) &\longrightarrow \mathbb{C}\\
 (E_G,\lambda,\nabla) &\longmapsto \lambda.
 \end{split}
	\end{align}
	
	The fiber $\pi^{-1}(0)$ over $0 \in \mathbb{C}$ is actually the moduli space of semistable $G$-Higgs bundles over $X$, i.e.
	\[
	\mathcal{M}^d_{\mathrm{Higgs}}(G) = \pi^{-1}(0) \subset \mathcal{M}^d_\mathrm{Hod}(G). 
	\]
	The natural $\mathbb{C}^*$-action (\ref{actionHiggs}) on $\mathcal{M}^d_{\mathrm{Higgs}}(G)$ extends to a $\mathbb{C}^*$-action on the Hodge moduli space $\mathcal{M}^d_\mathrm{Hod}(G)$ defined by 
	\begin{equation}\label{action}
	    t\cdot (E_G,\lambda,\nabla) \coloneqq (E_G,t\lambda,t\nabla).
	\end{equation}
	
	If we consider the case $\lambda = 1$, then the fiber $\pi^{-1}(1)$ is the moduli space $\mathcal{M}^d_{\mathrm{conn}}(G)$ of holomorphic $G$-connections on $X$.

\subsection{Semiprojectivity of the moduli space of $G$-Higgs bundles} Recall that the moduli space $\mathcal{M}^d_{\mathrm{Higgs}}(G)$ of $G$-Higgs bundles admits a standard $\mathbb{C}^*$-action
\[
t \cdot (E_G,\varphi) = (E_G,t\varphi),
\]
i.e. if $(E_G,\varphi)$ is semistable (resp. stable) then $(E_G,t\varphi)$ is semistable (resp. stable) for all $t\in \mathbb{C}^*$.

Let $\mathrm{rank}(G)=r$. Then the Hitchin map is given by
\begin{align*}
    h : \mathcal{M}^d_{\mathrm{Higgs}}(G) &\longrightarrow \mathcal{H} \coloneqq \bigoplus_{i=1}^{r}\mathrm{H}^0(X, K_X^{d_i})\\
    (E_G,\varphi) &\mapsto (p_1(\varphi),\dots, p_r(\varphi))
\end{align*}
where $\{p_1,\dots,p_r\}$ is a homogeneous basis for the ring of invariant polynomials on $\mathrm{Lie}(G)=\mathfrak{g}$ and $d_i$'s are degrees of $p_i$'s.
\begin{lemma}\label{equivariant}
The Hitchin map $h : \mathcal{M}^d_{\mathrm{Higgs}}(G) \to \mathcal{H}$ is $\mathbb{C}^*$-equivariant.
\end{lemma}
\begin{proof}
The Hitchin base $\mathcal{H}$ admits a standard $\mathbb{C}^*$-action which is given by
\[
t\cdot (s_1,s_2,\dots , s_r) = (t^{d_1}s_1,t^{d_2}s_2,\dots , t^{d_r}s_r).
 \]
Let $h(E_G,\varphi) = (s_1,s_2,\dots , s_r)$. Then, 
\begin{align*}
    h(t\cdot (E_G,\varphi)) &=h(E_G,t\varphi)\\
    &=(p_1(t\varphi),\dots, p_r(t\varphi))\\
    &=(t^{d_1}p_1(\varphi),t^{d_2}p_2(\varphi),\dots , t^{d_r}p_r(\varphi))\\
    &=(t^{d_1}s_1,t^{d_2}s_2,\dots , t^{d_r}s_r)\\
    &= t\cdot (s_1,s_2,\dots , s_r)\\
    &= t\cdot h(E_G,\varphi).
\end{align*}

Hence, $h$ is $\mathbb{C}^*$-equivariant.
\end{proof}

To prove the semiprojectivity of the moduli space, we need to show that the moduli space $\mathcal{M}^d_{\mathrm{Higgs}}(G)$ satisfies the conditions in the definition of the semiprojective variety as given in the introduction.
\begin{lemma}\label{limit}
Let $(E_G,\varphi) \in \mathcal{M}^d_{\mathrm{Higgs}}(G)$ be a semistable $G$-Higgs bundle. Then the limit $\lim_{t\to 0} (E_G,t\varphi)$ exists in $\mathcal{M}^d_{\mathrm{Higgs}}(G)$.
\end{lemma}
\begin{proof}
Consider the morphism
\[
f: \mathbb{C}^* \longrightarrow \mathcal{M}^d_{\mathrm{Higgs}}(G)
\]
given by $t \mapsto (E_G,t\varphi)$. Since $h$ is $\mathbb{C}^*$-equivariant (by \ref{equivariant}), we have
\[
\lim_{t\to 0} h(E_G,t\varphi) = \lim_{t \to 0} t \cdot h(E_G,\varphi) = 0.
\]
Thus, the composition map $F \coloneqq h \circ f : \mathbb{C}^* \longrightarrow \mathcal{H}$ extends to a morphism $\hat{F} : \mathbb{C} \longrightarrow \mathcal{H}$. By valuative criterion of properness (since $h$ is proper) $f$ extends to a morphism \[
\hat{f}: \mathbb{C} \longrightarrow \mathcal{M}^d_{\mathrm{Higgs}}(G).
\]
Hence, $\lim_{t\to 0} (E_G,t\varphi)$ exists in $\mathcal{M}^d_{\mathrm{Higgs}}(G)$.
\end{proof}

\begin{lemma}\label{fixed}
The fixed point locus under the $\mathbb{C}^*$-action on $\mathcal{M}^d_{\mathrm{Higgs}}(G)$ is proper in $h^{-1}(0)$.
\end{lemma}
\begin{proof}
Note that the origin is the only point on the Hitchin base $\mathcal{H}$ which is fixed under the $\mathbb{C}^*$-action. Thus, the fixed point subvariety $\mathcal{H}^{\mathbb{C}^*}$ is the singleton set $\{0\}$. Since $h$ is $\mathbb{C}^*$-equivariant, the fixed point locus $\mathcal{M}^d_{\mathrm{Higgs}}(G)^{\mathbb{C}^*}$ must be closed in $h^{-1}(\mathcal{H}^{\mathbb{C}^*}) = h^{-1}(0)$. Also, since $h$ is proper, so is $h^{-1}(0)$. Hence, $\mathcal{M}^d_{\mathrm{Higgs}}(G)^{\mathbb{C}^*}$ is proper in $h^{-1}(0)$.
\end{proof}
\begin{theorem}\label{Higgs}
The moduli space $\mathcal{M}^d_{\mathrm{Higgs}}(G)$ of semistable $G$-Higgs bundles is a semiprojective variety.
\end{theorem}
\begin{proof}
Since the moduli space $\mathcal{M}^d_{\mathrm{Higgs}}(G)$ is a quasi-projective variety, semiprojectivity follows from the Lemma \ref{limit} and \ref{fixed}.
\end{proof}

\subsection{Semiprojectivity of principal Hodge moduli space}
Recall the $\mathbb{C}^*$-action on $\mathcal{M}^d_\mathrm{Hod}(G)$ given as in (\ref{action}).
\begin{lemma}\label{limit1}
Let $(E_G,\lambda,\nabla) \in \mathcal{M}^d_\mathrm{Hod}(G)$ be a $\lambda$-connection on $E_G$. Then the limit $$\lim_{t\to 0} (E_G,t\lambda,t\nabla)$$ exists in $\pi^{-1}(0) \subset \mathcal{M}^d_\mathrm{Hod}(G)$, where $\pi :  \mathcal{M}^d_\mathrm{Hod}(G) \longrightarrow \mathbb{C}$ is the projection map (\ref{proj}).
\end{lemma}
\begin{proof}
The proof is similar to \cite[Corollary 10.2]{S97}. Consider the following projections
\[
\pi_1 : X \times \mathbb{C}^* \longrightarrow X \hspace{0.3cm} \mathrm{and} \hspace{0.3cm} \pi_2 : X \times \mathbb{C} \longrightarrow \mathbb{C}.
\]
Now consider the $\mathbb{C}^*$-flat family over $\pi_2: X \times \mathbb{C} \longrightarrow \mathbb{C}$ given by
\[
(\mathcal{E},t\lambda,\nabla_{\pi_2}) \coloneqq (\pi_1^*E_G,t\lambda,t\pi_1^*\nabla)
\]
For any $t\ne 0$, we know that a principal $t\lambda$-connection $(E_G,t\lambda,t\nabla)$ is semistable if and only if $(E_G,\lambda,\nabla)$ is semistable. Therefore, the fibers of the above family are semistable for $t\ne 0$. Following \cite[Theorem 10.1]{S97}, there exist a $\mathbb{C}$-flat family $(\overline{\mathcal{E}},\overline{t\lambda},\overline{\nabla_{\pi_2}})$ over $\pi_2 : X \times \mathbb{C} \longrightarrow \mathbb{C}$ such that 
\[
\left.(\overline{\mathcal{E}},\overline{t\lambda},\overline{\nabla_{\pi_2}})\right|_{X\times \mathbb{C}^*} \cong (\pi_1^*E_G,t\lambda,t\pi_1^*\nabla)
\]
and $\left.(\overline{\mathcal{E}},\overline{t\lambda},\overline{\nabla_{\pi_2}})\right|_{X\times \{0\}}$ is semistable. Therefore,
\[
\left.(\overline{\mathcal{E}},\overline{t\lambda},\overline{\nabla_{\pi_2}})\right|_{X\times \{0\}} \in \pi^{-1}(0)
\]
is the limit of the $\mathbb{C}^*$-orbit of $(E_G,\lambda,\nabla)$ at $t=0$ in the moduli space $\mathcal{M}^d_\mathrm{Hod}(G)$.
\end{proof}

\begin{lemma}\label{fixed1}
The fixed point subvariety $\mathcal{M}^d_\mathrm{Hod}(G)^{\mathbb{C}^*}$ of $\mathcal{M}^d_\mathrm{Hod}(G)$ is proper in $\mathcal{M}^d_\mathrm{Hod}(G)$.
\end{lemma}
\begin{proof}
The $\mathbb{C}^*$-action $\mathcal{M}^d_\mathrm{Hod}(G)$ is given by 
\[
t\cdot (E_G,\lambda,\nabla) = (E_G,t\lambda,t\nabla).
\]
Thus the fixed point subvariety is exactly same as the fixed point subvariety under the $\mathbb{C}^*$-action on $\pi^{-1}(0) = \mathcal{M}^d_{\mathrm{Higgs}}(G)$. Hence by Lemma \ref{fixed}, the fixed point subvariety $$\mathcal{M}^d_{\mathrm{Hod}}(G)^{\mathbb{C}^*} \subset \mathcal{M}^d_\mathrm{Hod}(G)$$ is proper.
\end{proof}

\begin{theorem}\label{Hodge}
The moduli space $\mathcal{M}^{d}_\mathrm{Hod}(G)$ is a semiprojective variety over $\mathbb{C}$.
\end{theorem}
\begin{proof}
    This follows from Lemma \ref{limit1} and Lemma \ref{fixed1}.
\end{proof}

\section{Cohomological and motivic consequences in the smooth case}

In this section we record structural consequences of the
$\mathbb{C}^\ast$-action on $M^d_{\mathrm{Higgs}}(G)$.
Since the classical Białynicki--Birula theorem requires smoothness,
we impose the additional hypothesis that
$M^d_{\mathrm{Higgs}}(G)$ is smooth. This holds, for example, when $G = GL_n(\mathbb{C})$  and
$\gcd(n,d)=1$. Under this assumption, $M^d_{\mathrm{Higgs}}(G)$ is a smooth
quasi-projective semiprojective variety.

\subsection*{Białynicki--Birula decomposition}

Let
\[
F = \left(M^d_{\mathrm{Higgs}}(G)\right)^{\mathbb{C}^\ast}
\]
be the fixed point locus of the $\mathbb{C}^\ast$-action.
By semiprojectivity, $F$ is a proper closed subvariety.
Write
\[
F = \bigsqcup_{i\in I} F_i
\]
for the decomposition into connected components.

For each component $F_i$, define the attracting set
\[
S_i =
\left\{
x \in M^d_{\mathrm{Higgs}}(G)
\;\middle|\;
\lim_{t\to 0} t\cdot x \in F_i
\right\}.
\]

Since $M^d_{\mathrm{Higgs}}(G)$ is semiprojective,
the limit exists for every point.
By the Białynicki--Birula theorem for smooth
semiprojective varieties (see
\cite{BB73, HT03}),
each $S_i$ is a locally closed smooth subvariety and
the morphism
\[
\pi_i : S_i \longrightarrow F_i,
\qquad
x \longmapsto \lim_{t\to 0} t\cdot x,
\]
is an algebraic vector bundle.

For $x \in F_i$, the fiber dimension equals the
dimension of the positive weight subspace in the
weight decomposition of the Zariski tangent space
$T_x M^d_{\mathrm{Higgs}}(G)$.
We denote this integer by $n_i$.

Moreover,
\[
M^d_{\mathrm{Higgs}}(G)
=
\bigsqcup_{i\in I} S_i
\]
is a disjoint decomposition.

\subsection*{Cohomological consequence}

We now describe the consequence for compactly supported
cohomology with complex coefficients.

Let $E \to B$ be an algebraic vector bundle of rank $n$
between complex quasi-projective varieties.
Then compactly supported cohomology satisfies
\[
H_c^k(E,\mathbb{C})
\cong
H_c^{\,k-2n}(B,\mathbb{C}),
\]
since the fiber $\mathbb{A}^n$ has compactly supported
cohomology concentrated in degree $2n$.

\begin{theorem}
Assume $M^d_{\mathrm{Higgs}}(G)$ is smooth.
Then there is an isomorphism of graded vector spaces
\[
H_c^k\!\left(M^d_{\mathrm{Higgs}}(G),\mathbb{C}\right)
\cong
\bigoplus_{i\in I}
H_c^{\,k-2n_i}(F_i,\mathbb{C}).
\]
\end{theorem}

\begin{proof}
Since
\[
M^d_{\mathrm{Higgs}}(G)
=
\bigsqcup_{i\in I} S_i
\]
is a disjoint union of locally closed subvarieties,
compactly supported cohomology is additive:
\[
H_c^k\!\left(M^d_{\mathrm{Higgs}}(G)\right)
=
\bigoplus_{i\in I}
H_c^k(S_i).
\]

Each $S_i \to F_i$ is a vector bundle of rank $n_i$.
Hence
\[
H_c^k(S_i)
\cong
H_c^{\,k-2n_i}(F_i).
\]
Combining these isomorphisms yields the result.
\end{proof}

In particular, the compactly supported cohomology of
$M^d_{\mathrm{Higgs}}(G)$ is completely determined by
the cohomology of the fixed components together with
the integers $n_i$.

\subsection*{Motivic consequence}

We now describe the analogous statement at the level of
the Grothendieck ring of varieties.

Let $K(\mathrm{Var}_{\mathbb{C}})$ denote the Grothendieck
ring of complex quasi-projective varieties.
It is generated by symbols $[Z]$ for quasi-projective
varieties $Z$, subject to the scissor relation
\[
[Z] = [U] + [Z \setminus U]
\]
for every open subvariety $U\subset Z$.
Multiplication is defined by
\[
[Z_1]\cdot[Z_2] = [Z_1\times Z_2].
\]

The class
\[
\mathbb{L} = [\mathbb{A}^1]
\]
is called the Lefschetz motive.
If $E \to B$ is a vector bundle of rank $n$, then
\[
[E] = \mathbb{L}^{\,n}[B].
\]

\begin{theorem}
Assume $M^d_{\mathrm{Higgs}}(G)$ is smooth.
Then in the Grothendieck ring of varieties,
\[
\left[
M^d_{\mathrm{Higgs}}(G)
\right]
=
\sum_{i\in I}
\mathbb{L}^{\,n_i}
\,[F_i].
\]
\end{theorem}

\begin{proof}
Using additivity in the Grothendieck ring,
\[
\left[
M^d_{\mathrm{Higgs}}(G)
\right]
=
\sum_{i\in I}
[S_i].
\]
Since each $S_i \to F_i$ is a vector bundle of rank $n_i$,
we have
\[
[S_i] = \mathbb{L}^{\,n_i}[F_i].
\]
Substituting gives the formula.
\end{proof}

\section*{Acknowledgment}
The first author is supported by the INSPIRE faculty fellowship (Ref No.: IFA22-MA 186) funded by the DST, Govt. of India. The second named author is partially supported by SERB SRG Grant SRG/2023/001006.


\begin{thebibliography}{30}
	   \bibitem{AO21}
	   D. Alfaya and A. Oliveira,
	   \emph{Lie algebroid connections, twisted Higgs bundles and motives of moduli spaces}, Journal of Geometry and Physics, Volume 201 (2024), 105195.
		

	\bibitem{A57}
		M. F. Atiyah
		\emph{Complex analytic connections in fibre bundles.} Trans. Amer. Math. Soc. 85 (1957), 181-207.
		
		\bibitem{AB02}
		H. Azad and I. Biswas,
		\emph{On holomorphic principal bundles over a compact
			Riemann surface admitting a flat connection}, Math. Ann. 322, 333–346 (2002)

            \bibitem{BB73}
A. Białynicki-Birula,
\emph{Some theorems on actions of algebraic groups},
Ann. of Math. (2) 98 (1973), 480–497.
		
		\bibitem{BGH13}
		I. Biswas, T. Gómez and N. Hoffmann, 
		\emph{Torelli theorem for the Deligne-Hitchin moduli space, II}, Doc. Math. 18 (2013) 1177–1189
		
		\bibitem{BH12}
		 I. Biswas and N. Hoffmann, 
		 \emph{A Torelli theorem for moduli spaces of
principal bundles over a curve}, Ann. Inst. Fourier 62 (2012), 87–106.






\bibitem{HT03}
	 T. Hausel and M. Thaddeus, 
	 \emph{Mirror symmetry, Langlands duality, and the Hitchin system}.
Invent. Math., 153(1):197–229, (2003).
	  

	   

    

 \bibitem{H87}
 N. J. Hitchin, 
 \emph{Stable bundles and integrable systems}, Duke Math. Jour.
54 (1987), 91–114.

 \bibitem{R75}
 A. Ramanathan, 
 \emph{Stable principal bundles on a compact Riemann surface}, Math. Ann. 213 (1975), 129–152.



\bibitem{R96}
 A. Ramanathan,
 \emph{Moduli for principal bundles over algebraic curves I-II}, Proc. Indian Acad. Sci. Math. Sci. 106(1996), 301-28, 421-49.

 \bibitem{S92}
 C. T. Simpson, 
 \emph{Higgs bundles and local systems}, Inst. Hautes ´ Etudes
Sci. Publ. Math. 75 (1992), 5–95.

   \bibitem{S94a}
      C. T. Simpson, 
      \emph{Moduli of representations of the fundamental group of a smooth projective variety I}. Publi. Math. I.H.E.S., 79:47–129, 1994.
 
\bibitem{S94}
C. T. Simpson,
\emph{Moduli of representations of the fundamental group of a smooth projective variety. II}, Inst. Hautes Etudes Sci. Publ. Math. 80 (1994), 5-79

	  
	  \bibitem{S97}
	  C. T. Simpson, 
	  \emph{The Hodge filtration on nonabelian cohomology}. Proceedings of Symposia in Pure Mathematics, 62.2:217–284, 1997.

    
	   
	   

	   
	   	
           
		
		



\end{thebibliography}
\end{document}